\documentclass[10pt]{amsart}
\usepackage{amsmath,amssymb,enumerate}
  
  \date{\today}
  
\newcommand{\ds}{\displaystyle}

\renewcommand{\r}{\mathbb{R}}
\DeclareMathOperator{\vol}{vol}%

\newtheorem{theorem}{\rm\bf Theorem}
\newtheorem*{atheorem}{\rm\bf Duality Theorem}
\newtheorem{proposition}[theorem]{\rm\bf Proposition}
\newtheorem{lemma}[theorem]{\rm\bf Lemma}

\theoremstyle{definition}
\newtheorem{definition}[theorem]{\rm\bf Definition}

\title{A short proof of the H\"older-Poincar\'{e} Duality for $L_{p}$-cohomology}

\author{Vladimir Gol'dshtein} 
\address{Vladimir Gol'd'shtein, Department of Mathematics,  
Ben Gurion University of the Negev,  P.O.Box 653, Beer Sheva, Israel}  
\email{vladimir@bgumail.bgu.ac.il}

\author{Marc Troyanov} 
\address{M. Troyanov, 
Section de MathŽmatiques,  
\'Ecole Polytechnique F{\'e}derale de Lausanne, 
1015 Lausanne - Switzerland}
\email{marc.troyanov@epfl.ch}

\begin{document}
\begin{abstract}
We give a short proof of the duality theorem for the reduced $L_p$-cohomology
of  a complete oriented  Riemannian manifold.

\medskip

\end{abstract}
\maketitle


Let $(M,g)$ be a Riemannian manifold. For any $1\leq p < \infty$ we denote by 
$L^{p}(M,\Lambda^{k})$ the space of $p$-integrable differential forms on $M$.
An element of that space is a measurable differential $k$-forms $\omega$
such that
 \[
\Vert\omega\Vert_{p}:=\left(\int_{M}|\omega|_{x}^{p}d\vol_{g}(x)\right)^{1/p}<\infty.
\]
Let $Z^k_p(M) = \ker d \cap L^p(M,\Lambda^k)$, this is the set of weakly
closed forms in $L^p(M,\Lambda^k)$  and
\[
  B^k_{p}(M) = d\left(L^{p}(M,\Lambda^{k-1}) \right) \cap
  L^p(M,\Lambda^k).
\]
We also denote by  $\overline{B}_{p}^{k}(M)$  the closure of
$B_{p}^{k}(M)$ in $L^p(M,\Lambda^k)$.
Because $Z^k_p(M) \subset  L^p(M,\Lambda^k)$ is a closed subspace and $d\circ d = 0$, we have 
$\overline{B}_{p}^{k}(M) \subset Z_{p}^{k}(M)$.
The reduced $L_{p}$\emph{-cohomology} of $(M,g)$
(where $1\leq p < \infty $) is defined to be the quotient
\begin{equation*}
\overline{H}_{p}^{k}(M) =Z_{p}^{k}(M)/\overline{B}_{p}^{k}(M).
\end{equation*}

\medskip

This is a Banach space and the  goal of this paper is to prove the following

\begin{atheorem}\label{th.duality1}
Let  $(M,g)$ be a complete oriented Riemannian manifold of dimension $n$ and $1<p<\infty$. Then 
$\overline{H}_{p}^{k}(M)$ and $\overline{H}_{p'}^{n-k}(M)$ are dual to each other. The duality is
given by the integration pairing:
$$
 \begin{array}{cccc}
 \overline{H}_{p}^{k}(M) \times \overline{H}_{p'}^{n-k}(M) & \to & \quad \r \\ \\
([\omega] , [\theta])  &\mapsto &  \int_M\omega\wedge\theta. 
 \end{array} 
$$
\end{atheorem}

\medskip

\textbf{Remark} The result has been obtained in 1986 by 
V. M. Gol'dshtein, V.I. Kuz'minov and I.A.Shvedov, see \cite{GK4}. In fact that paper
also describes the dual space to the $L_p$-cohomology of non complete manifolds. 
The proof we present here is simpler and more direct than the proof in  \cite{GK4}, although
it doesn't seem to be extendable to the non complete case. Note that this duality theorem is 
useful to prove vanishing or non vanishing results in $L_p$-cohomology, see e.g.
\cite{Gromov93,Pansu2008,GT2006}.

Let us also mention that Gromov deduced the above theorem from the simplicial version of
the $L_p$-cohomology, see \cite{Gromov93}. Gromov's argument works only for Riemannian manifolds
with bounded geometry, while the proof we give here works for any complete manifold.
Our proof can also be extended to the more general $L_{q,p}$-cohomology, see \cite{qpduality}.

\medskip  

The proof will rest on a few auxiliary facts.
We will first need a description of the dual space to $L^{p}(M,\Lambda^{k})$, see \cite{GK4}: 
\begin{proposition}
If $1 < p<\infty$ and $p'=p/(p-1)$, then the  pairing 
$L^{p}(M,\Lambda^{k})\times L^{p'}(M,\Lambda^{n-k})\rightarrow\r$
defined by
\begin{equation} \label{Ipairing}
  \left\langle \omega,\varphi\right\rangle =\int_{M}\omega\wedge\varphi
\end{equation}
is continuous and non degenerate.
\end{proposition}

\medskip
 
We will also need the following density result whose  proof is based on  regularization methods,  see e.g. \cite{GK3,GT2006}:
 \begin{proposition}\label{prop.density}
Let $\theta \in  L^p(M,\Lambda^{k-1})$ be a $(k-1)$-form whose weak exterior differential is $p$-integrable,  $d\theta \in  L^p(M,\Lambda^{k})$. Then there exists a sequence
$\theta_{j}\in C^{\infty}(M,\Lambda^{k-1})$ such that  $\ds \theta=\lim_{j\rightarrow\infty}\theta_{j}$ 
and $\ds d\theta=\lim_{j\rightarrow\infty}d\theta_{j}$ in $L^{p}$. 
\end{proposition}

\medskip

The next lemma is the place where the completeness hypothesis enters.
\begin{lemma}\label{lem.ZB}
 If $(M,g)$ is complete, then $d\mathcal{D}^{k-1}(M)$ is dense in $B^k_{p}(M)$.
\end{lemma}

\begin{proof}
Because $M$ is complete, one can find  a sequence
of smooth functions with compact support \(\{\eta_j\} \subset C^{\infty}_0(M)\) such that 
 \(0 \leq \eta_j \leq 1\), \  \( \lim_{j\to \infty} \sup |d\eta_j| = 0\) and 
 \(\eta_j \to 1\) uniformly on every compact subset of \(M\).
Let $\omega\in B_{p}^{k}(M)$, then there
exists $\theta \in  L^p(M,\Lambda^{k-1})$ such that $d\theta=\omega$. Choose a sequence 
$\{\theta_{j}\} \subset  C^{\infty}(M,\Lambda^{k-1})$ as in Proposition
\ref{prop.density}  and set
$\tilde{\theta}_{j} = \eta_j \theta_{j} \in \mathcal{D}^{k-1}$. We then have
$$
 \| \tilde{\theta}_{j} - {\theta}_{j}  \|_p =  \| (\eta_j -1){\theta}_{j}  \|_p \to 0
$$
and
\begin{align*}
 \|d \tilde{\theta}_{j} - d{\theta}_{j}  \|_p  & \leq  \| (\eta_j -1)d{\theta}_{j}  \|_p 
 +  \| d\eta_j \wedge{\theta}_{j}  \|_p 
\\  &  
 \leq  \| (\eta_j -1)d{\theta}_{j}  \|_p 
 +  \sup |d\eta_j| \cdot \|{\theta}_{j}  \|_p 
  \to 0.
\end{align*}
This implies that $\ds \omega =\lim_{j\rightarrow\infty}d\tilde{\theta}_{j}$ in $L^{p}$.
\end{proof}

\begin{definition}
A duality between two reflexive Banach spaces $X_{0}$,$X_{1}$ is a non degenerate continuous bilinear map  \(I:X_{0}\times X_{1}\to \mathbb{R}\).
A duality naturally induces an isomorphism between $X_{1}$ and the dual of $X_{0}$.

Given such a duality and  a nonempty subset $B$ of $X_{0}$, we define the \emph{annihilator} $B^{\bot} \subset X_1$ of $B$ to be the set of all elements
$\eta \in  X_1$ such that $I(\xi, \eta) = 0$ for all $\xi \in B$.
\end{definition}

Recall few main properties of annihilators. For any $B \subset X_0$ the annihilator $B^{\bot}$ is a closed linear subspace of   $X_1$. The Hahn-Banach theorem implies that if  $B$ is a linear subspace of $X_0$ then $({B^{\bot}})^{\bot}=\overline{B}$. 

For these and further facts on the notion of annihilator, we refer to the book
\cite{Brezis} or \cite{conway90}.

\medskip

The proof of the duality Theorem is based on the following lemma  about annihilators:

\begin{lemma}\label{lem.dualquotient}
Let \(I:X_{0}\times X_{1}\to \mathbb{R}\) be a duality between two reflexive Banach spaces. Let 
\(B_{0},A_{0},B_{1},A_{1}\) be linear subspaces  
such that
\[
 B_{0}\subset A_{0} = B_{1}^{\bot} \subset X_{0}
       \qquad  \text{and} \qquad  
B_{1}\subset A_{1} = B_{0}^{\bot}\subset X_{1}.
\]
Then the pairing \(\overline{I}:\overline{H}_{0}\times\overline{H}_{1}\to \mathbb{R}\)
of \(\overline{H}_{0}:=A_{0}/\overline{B}_{0}\) and \(\overline{H}_{1}:=A_{1}/\overline{B}_{1}\)
is well defined and induces a duality between \(\overline{H}_{0}\) and
\(\overline{H}_{1}\).
\end{lemma}

\begin{proof}
Observe first that $A_{i} \subset X_{i}$ is a closed subspace since 
the annihilator of any subset of a Banach space is always a closed linear subspace.

The bounded bilinear map $I:A_{0}\times A_{1}\rightarrow\mathbb{R}$
is defined by restriction. It gives rise to a well defined bounded
bilinear map $\overline{I}:A_{0}/\overline{B}_{0}\times A_{1}/\overline{B}_{1}\rightarrow\mathbb{R}$
because we have the inclusions $B_{0}\subset A_{0}\subset B_{1}^{\bot}$
and $B_{1}\subset A_{1}\subset B_{0}^{\bot}$.

We show that $\overline{I}$ is non degenerate: let $a_{0}\in A_{0}$
be such that $[a]\neq0\in A_{0}/\overline{B}_{0}$; i.e. $a\not\in\overline{B}_{0}$.
By Hahn-Banach theorem and the fact that $X_{1}$ is dual to $X_{0}$, there
exists an element $y\in X_{1}$ such that $I(a,y)\neq0$ and $I(b,y)=0$
for all $b\in\overline{B}_{0}$. Thus $y\in B_{0}^{\bot}=A_{1}$ and
we have found an element $[y]\in A_{1}/\overline{B}_{1}$ such that
$\overline{I}([a],[y])\neq0$. \ 
The same argument shows that for any $[\alpha]\neq0\in A_{1}/\overline{B}_{1}$,
we can find an element $[x]\in A_{0}/\overline{B}_{0}$ such that
$\overline{I}([x],[\alpha])\neq0$. 
\end{proof}

\medskip
 
{\bf Proof of the Main Theorem.}
Let $\phi \in  L^p(M,\Lambda^k)$, then $d\phi= 0$ in the weak sense if 
and only if $\int_M \phi \wedge d\omega = 0$ for any  $\omega \in\mathcal{D}^{n-k-1}$.
This  precisely means that  $Z^k_p(M)\subset L^p(M,\Lambda^k)$ is the annihilator
of $d\mathcal{D}^{n-k-1}
  \subset L^{p'}(M,\Lambda^{n-k})$ for the pairing (\ref{Ipairing}):  
$$
    Z^k_p(M) = (d\mathcal{D}^{n-k-1})^{\perp}.
$$
By lemma \ref{lem.ZB}, $d\mathcal{D}^{n-k-1}$ and ${B}^{n-k}_{p'}$ have the same annihilator,
thus 
\[
    {B}^k_{p}  \subset  Z^k_{p} =  ({B}^{n-k}_{p'})^{\perp} \subset  L^{p}(M,\Lambda^{k}).
\]  
Similarly, we also have  
\[
   {B}^{n-k}_{p'} \subset  Z^{n-k}_{p'} =  ({B}^k_{p})^{\perp} \subset L^{p'}(M,\Lambda^{n-k}),
\]
and  Lemma \ref{lem.dualquotient}, says that the duality  (\ref{Ipairing})
 induces a duality between
$Z^{n-k}_{p'}/\overline{B}^{n-k}_{p'}$ and $Z^k_{p}/\overline{B}^k_{p}$.
\qed



\begin{thebibliography}{10}

\bibitem{Brezis}H. Brezis \emph{Analyse fonctionnelle, Th\'{e}orie et applications.}
Dunod, Paris 1999. 

\bibitem{conway90}J. Conway \emph{A course in functional analysis.} Second edition.
Graduate Texts in Mathematics, 96. Springer-Verlag, New York, 1990 .

\bibitem{GK3} V. M. Gol'dshtein, V.I. Kuz'minov, I.A.Shvedov \emph{A Property of
de Rham Regularization Operators} Siberian Math. Journal, \textbf{25},
No 2 (1984). 

\bibitem{GK4} V. M. Gol'dshtein, V.I. Kuz'minov, I.A.Shvedov \emph{Dual spaces of
Spaces of Differential Forms} Siberian Math. Journal, \textbf{54},
No 1 (1986). 

\bibitem{GT2006} V. Gol'dshtein and M. Troyanov. \emph{Sobolev Inequality for
Differential forms and $L_{p}$-cohomology.} Journal of Geom. Anal.,
\textbf{16,} No 4, (2006), 597-631. 

\bibitem{qpduality} V. Gol'dshtein and M. Troyanov. 
\emph{The H\"older-Poincar\'{e} Duality for $L_{q,p}$-cohomology} Preprint.

\bibitem{Gromov93}  M. GROMOV,
{\sl Asymptotic invariants of infinite groups.}
In ``Geometric Group Theory'', ed. G. Niblo and M. Roller, Cambridge: Cambridge University Press, $(1993)$.

\bibitem{Pansu2008} P. Pansu, 
\emph{Cohomologie Lp et pincement.}
Comment. Math. Helv. 83, 327Ð357 (2008).

\end{thebibliography}
\end{document}